\documentclass[ 10pt,reqno]{amsart}

\usepackage{cite}
\usepackage{setspace}
\usepackage[utf8]{inputenc}
\usepackage{amsmath}
\usepackage{amsthm}
\usepackage{amssymb}
\usepackage{amscd}
\usepackage{amsfonts}
\usepackage{latexsym}
\usepackage[hidelinks]{hyperref}
\usepackage{cite}
\usepackage{dirtytalk}
\usepackage{xcolor}
\usepackage[headings]{fullpage}
\usepackage{longtable}
\usepackage{enumitem}
\usepackage{comment}
\usepackage{multirow}
\usepackage{fancyhdr}
\usepackage{algpseudocode}

\theoremstyle{plain}
\newtheorem{theorem}{Theorem}
\newtheorem{corollary}{Corollary}
\newtheorem{lemma}{Lemma}
\newtheorem{algorithm}{Algorithm}

\newtheorem{fact}{Number Theory Fact}
\newtheorem{definition}{Definition}

\newtheorem*{remark}{Remark}

\newcommand{\lp}{\left}
\newcommand{\rp}{\right}
\newcommand{\ra}{\rightarrow}

\newcommand{\NN}{\mathcal{N}}

\newcommand{\BB}{\mathcal{B}}
\newcommand{\MM}{\mathcal{M}}

\newcommand{\PP}{\mathcal{P}}

\newcommand{\EE}{\mathcal{E}}

\newcommand{\E}{\mathbb{E}}
\newcommand{\I}{\mathbb{I}}
\newcommand{\R}{\mathbb{R}}
\newcommand{\N}{\mathbb{N}}
\newcommand{\F}{\mathbb{F}}

\renewcommand{\P}{\mathbb{P}}

\begin{document}

\title{On the distribution of very short character sums}

\author[P.~Nosal]{Pawe{\l}~Nosal}
\address{University of Warwick, Mathematics Institute, Zeeman Building, University of Warwick, Coventry CV4 7AL}
\email{Pawe.Nosal@warwick.ac.uk}
\urladdr{\url{https://sites.google.com/view/pawelnosalmaths}}

\thanks{The author is supported by the Warwick Mathematics Institute Centre for Doctoral Training, and gratefully acknowledges funding from the Heilbronn Institute for Mathematical Research (HIMR) and the UK Engineering and Physical Sciences Research Council under the "Additional Funding Programme for Mathematical Sciences" (Grant number: EP/V521917/1)}

\maketitle

\begin{abstract}
    We establish a central limit theorem of $\tfrac{1}{\sqrt{h_p}}\sum_{X< n \leq X+h_p}\lp(\tfrac{n}{p}\rp)$ for almost all the primes $p$, with $X$ uniformly random in $[g(p)]$, $g(p)$ an arbitrary divergent function growing slower than any power of $p$, provided $(\log h_p)/(\log g(p))\ra 0, h_p \ra \infty$ as $p \ra \infty$. This improves the recent results of Basak, Nath and Zaharescu, who established this for $g(p) = (\log p)^A, A>1$. We also use the best currently available tools to expand the original central limit theorem of Davenport and Erd\H{o}s for all the primes to a shorter interval of starting points.\\
    In this paper we exploit a Selberg's sieve argument, recently used by Harper, an intersection result due to Evertse and Silverman and some consequences of the Weil bound on general character sums. 
\end{abstract}

\section{Introduction}
Let $q$ be a prime number. Obviously, the least quadratic residue modulo any prime is  $1$. No such information is in general offered for quadratic non-residues (further denoted NQR) and the only trivial observation that we can make is that the first NQR needs to be a prime smaller than $q$ (by complete multiplicativity of Dirichlet characters and their orthogonality relation). Detection of the least NQR is one of the main motivations for study of sums of the type 
\begin{equation*}
    \sum_{n \leq h}\chi_q(n),
\end{equation*}where $\chi_q$ is the Dirichlet character realized by the Legendre symbol $(\tfrac{n}{q})$. Being able to non-trivially bound short character sums immediately gives us a way to detect the least quadratic non-residues, simply by finding the first $h$ when $$\sum_{n \leq h}\chi_q(n)< \lfloor h \rfloor.$$In general, the study of partial sums of sequences has an absolutely critical role in analytic number theory, as via standard techniques, like Perron's formula, these are directly connected to the behaviour of their Dirichlet series.
In this paper the object of our attention are character sums with varying starting point $$S_h(x,q): = \sum_{x < n \leq x+h}\chi_q(n),$$ where $q$, the modulus of the character is huge and $h =o(q)$. Note that since $\chi(n)$ is periodic with period $q$, there is no need to consider sums with $h>q$. 

\subsection{Remark on notation}
Throughout the paper we adopt the usual asymptotic notation. We write $f \ll g$, in the Vinogradov notation, or $f = O(g)$ in the \textit{big O} notation to mean $|f(x)| \leq C g(x)$, for some constant $C>0,$ as $x \ra \infty$. Similarly, we write $f(x) = o(g(x))$ to signify that $|(f/g)(x)| \ra 0$ as $x \ra \infty$. 
For any $X \in \R, k \in \N_{\geq 0}$ we shall also write 
\begin{equation*}
    [X] : = \{1,2,...,\lfloor X\rfloor\}
\end{equation*} and 
\begin{equation*}
    [X]^{k} : = \{ x \in \R^{k}\text{ with coefficients in } [X]\}.
\end{equation*}Additionally, for convenience of comparison we establish notation closely resembling that of \cite{BasakNathZaharescu}. Let $\PP$ be the set of primes, $\eta \geq 0.01 $ be an arbitrary constant, $h_q$ depends on a prime $q \in \PP$. Given a subset of primes $\BB \subseteq \PP$ we say it is $\eta$-sufficient if  
    \begin{equation}\label{eta_strong}
        \lp|(\PP\backslash\BB)\cap[X, X+X^{\eta}]\rp|=o\lp( \frac{X^{\eta}}{\log X}\rp),
    \end{equation}holds for all $X$ sufficiently large. We will say $\BB$ is $\eta$-strong if for some $\delta>0$
    \begin{equation*}
         \lp|(\PP\backslash\BB)\cap[X, X+X^{\eta}]\rp| \leq \frac{X^{\eta}}{(\log X)^{1+\delta}}
    \end{equation*}holds for all $X$ sufficiently large. Now, let $g(q)$ be any function that diverges to infinity. Further, for any real number $\lambda$ define
    \begin{equation*}
        \MM_{g, q, h_q}(\lambda) : = \frac{1}{g(q)}\lp|1 \leq m \leq g(q); S_{h_q}(m,q) \leq \lambda \sqrt{h_q} \rp|.
    \end{equation*}Finally, we define the extended Rademacher random multiplicative function:
\begin{definition}{Extended Rademacher random multiplicative functions}\\
Let $(f(q))_{q \in \PP}$ be a set of independent random variables indexed by primes, such that for any $q \in \PP, f(q) = \pm 1$ with equal probability. Then, for any $n \in \N$ define 
\begin{equation*}
    f(n) : = \prod_{q^{\alpha_q}\mid \mid n }f(q)^{\alpha_q},
\end{equation*} and $f(1) = 1$, so that $f(n)$ is completely multiplicative. We call $f(n)$ the extended Rademacher random multiplicative function. 
\end{definition}

\subsection{Historical results and motivation}
The first non-trivial result on distribution of character sums was proven independently by P\'olya and Ivan Vinogradov:
\begin{equation}\label{polya_vinogradov}
    \sum_{n \leq h}\chi(n) \ll \sqrt{q}\log q,
\end{equation}where $\chi$ is any non-principal Dirichlet character of conductor $q$. In the paper of Vinogradov, he also came up with a method, which provided a character sum bound for Legendre symbols like the one above, allows us to get an additional saving for the size of the least NQR. He showed that the least NQR is of order $$\ll q^{\frac{1}{2\sqrt{e}}}(\log q)^2$$ and conjectured that its actual size is 
\begin{equation}\label{Vinogradov's_conjecture}
    \ll_{\epsilon} q^{\epsilon}, 
\end{equation}for all $\epsilon>0$. Following this, the next improvement was due to Davenport and Erd\H{o}s \cite{Davenport2022TheDO}, who removed a $2-1/\sqrt{e}$ factors of $\log q$ from Vinogradov's bound of least quadratic NQR. After that, Burgess \cite{Burgess_1957} improved the character sum bound to 
\begin{equation*}
    \sum_{n \leq H}\chi(n) \ll q^{1/4}\log q
\end{equation*} and used Vinogradov's trick to obtain the last unconditional improvement of the general bound for the least NQR:
\begin{equation*}
    \ll_{\epsilon} q^{\frac{1}{4\sqrt{e}}+\epsilon},
\end{equation*}for any fixed $\epsilon$. Improving this bound remains an important problem in analytic number theory. Worth mentioning are the conditional results of Ankeny \cite{Ankeny1952TheLQ}, who showed that under the Generalized Riemann Hypothesis the least NQR is $\ll (\log q)^2$, and Lamzouri et al. \cite{Lamzouri_Li_Sound} who made this bound explicit with implied constant equal to one. However, in this paper we will be concerned with another type of results, first of which appeared in the aforementioned paper of Davenport and Erd\H{o}s \cite{Davenport2022TheDO}. They establish, that if $X$ is uniformly random $\in \{0,1,...,q-1\}$ and $$ \frac{\log h_q}{\log q} \ra 0, h_q \ra \infty$$ with $q$, then the normalized sums of lengths $h_q$ with a starting point at $X$ converge in distribution to a standard normal
\begin{equation}\label{distributional_convergence}
    \frac{\sum_{X < n \leq X+h_q}\chi_q(n)}{\sqrt{h_q}} \ra_D \NN(0,1),
\end{equation}as the conductor $q \ra \infty$ and $\chi_q$ is the Legendre symbol. This type of central limit theorem results are the main focus of our paper. In this setting, short character sums were studied by Chan, Choi and Zaharescu \cite{Chan2003AMV} who established a multidimensional version of \eqref{distributional_convergence}, Lamzouri \cite{Lamzouri2011TheDO}, who extended it to more general Dirichlet characters, Harper \cite{harper2022notecharactersumsshort}, who showed that for some extremal ranges of $H$,  \eqref{distributional_convergence} is true for almost all the primes, but \textbf{not} all the primes, and Basak, Nath and Zaharescu \cite{BasakNathZaharescu} who showed \eqref{distributional_convergence} for almost all the primes, with $X$ uniformly random in interval of length $(\log q)^A, A>1$. All of these results employ the method of moments to establish desired distributional convergence and it will be the same in our case. We note in passing, that long character sums with $h_q \approx q$ have been studied in depth for example in \cite{Bober_Goldmakher_Granville_Koukoulopoulos} or \cite{Granville_Sound1999LargeCS, Granville2015LargeCS,Granville_Sound2005LargeCS}. This paper was motivated by the aforementioned work of Basak, Nath and Zaharescu \cite{BasakNathZaharescu} on the distribution of short character sums with moving starting point and the paper of Harper \cite{harper2022notecharactersumsshort}. Authors of \cite{BasakNathZaharescu} managed to prove that if one takes the limit in equation \eqref{distributional_convergence} along any but a small, badly behaved set of primes, then one can prove suitable central limit theorem while having the starting point chosen uniformly at random from a set of length $(\log q)^A, A>1$. 

\begin{theorem}{(Basak-Nath-Zaharescu, 2023)}\label{BNZ_Theorem}\
Let $\eta \in (1/2, 1], A>1$ be arbitrary constants. For each prime $q$ choose an integer $h_q$ such that $h_q \ra \infty$ and $\tfrac{\log h_q}{\log \log q} \ra 0$ as $q\ra \infty$. Then, there exists an $\eta$-strong set $\BB$ of primes such that for any $\lambda \in \R$ 
\begin{equation*}
    \lim_{\substack{q \in \BB\\q \ra \infty}} \MM_{(\log q)^A,q,h_{q}}(\lambda) = \frac{1}{\sqrt{2\pi}}\int_{-\infty}^{\lambda}e^{-t^2/2}dt.
\end{equation*}
\end{theorem}In our work, we streamline the proof, extend their results, and additionally present an up to date central limit theorem result for all the primes. \newline

\subsection{Statement of results and discussion of improvements}
In the main result of the paper we improve on the above. 
\begin{theorem}\label{New_Theorem}
Let $g$ be an increasing, divergent function such that 
    \begin{itemize}
        \item $g'(c)=O\lp((g(Q))^{0.99}Q^{-0.01}\rp), \forall c \in [Q,Q+Q^{0.01}]$ as $Q \ra \infty$,
        \item $g(Q) \ll Q^{\epsilon}, \forall \epsilon>0$.
    \end{itemize}For each prime $q,$ choose $h_q$ such that $\log h_q /(\log  g(q))\ra 0$ and $h_q \ra \infty$ as $q \ra \infty$. Then, there exists a set $\BB$ of primes such that $\BB$ is $\eta$-sufficient for all $\eta \in [0.01,1]$ and for any $\lambda \in \R$, 
\begin{equation*}
    \lim_{\substack{q \ra \infty \\ q \in \BB}} \MM_{g, q, h_q}(\lambda) = \frac{1}{\sqrt{2\pi}}\int_{-\infty}^{\lambda}e^{-t^2/2}dt.
\end{equation*}
\end{theorem}Especially, we can choose $g(q) = (\log q )^A, 0< A\leq 1$, not permitted by Theorem \ref{BNZ_Theorem}. In this case, one can also recover from our proof that $\BB$ can be taken to be $\eta$-strong, so the above extends Theorem \ref{BNZ_Theorem} to this missing range. We would like to note that the lower bound for the parameter, $\eta = 0.01$ is just chosen for convenience. Indeed, we will see that one could conduct all of our proofs if we let $\eta= a, a\in (0,1]$ be any fixed positive number. After proving Theorem \ref{New_Theorem} one can use it to prove the following Corollary. 

\begin{corollary}\label{New_Corollary}
Let $g$ be an increasing, divergent function such that 
    \begin{itemize}
        \item $g'(c)=O\lp((g(Q))^{0.99}Q^{-0.01}\rp), \forall c \in [Q,Q+Q^{0.01}]$ as $Q \ra \infty$,
        \item $g(Q) \ll Q^{\epsilon}, \forall \epsilon>0$.
    \end{itemize}For each prime $q,$ choose $h_q$ such that $\log h_q /(\log  g(q))\ra 0$ and $h_q \ra \infty$ as $q \ra \infty$. Then, there exists a set $\BB$ of primes such that for any $\lambda \in \R$, 
\begin{equation*}
    \lim_{\substack{q \ra \infty \\ q \in \BB}} \MM_{g, q, h_q}(\lambda) = \frac{1}{\sqrt{2\pi}}\int_{-\infty}^{\lambda}e^{-t^2/2}dt
\end{equation*}and 
\begin{equation*}
    \lim_{X \ra \infty} \frac{\lp|\BB \cap [X, X+X^{\eta}\rp|]}{\lp|\PP \cap [X, X+X^{\eta}]\rp|} = 1, \text{ for all } \eta \in [0.525, 1].
\end{equation*}
\end{corollary}Similarly as with Theorem \ref{New_Theorem}, this recovers and extends on \cite[Corollary ~1.2]{BasakNathZaharescu}. The improvement to the method of \cite{BasakNathZaharescu} in our paper comes from an observation, that one can pass from an average of sums of Legendre symbols over primes in intervals of length $Q^{\eta}$, to expectation of these sums with Legendre symbol replaced with extended Rademacher random multiplicative functions. This fact, established by Harper in \cite[Number Theory Result~3]{harper2022notecharactersumsshort}  (for $\eta=1$), can be proved using appropriate Selberg's sieve weights. This allows us to save a factor of $(\log Q)$ lost in argument of \cite{BasakNathZaharescu}, clean up the analysis (as $f(n)$ are perfectly orthogonal, so nicer to work with than Legendre symbols), and allows much shorter ranges for the averages over primes ($\eta\leq0.5$). 
In the last part of the paper we gather some technical tools to extend the distributional results of \cite{Davenport2022TheDO}, to allow for the starting point to be chosen uniformly at random from $[g(q)],$ where $g(q)$ is any function satisfying $Q^{1/2}\log Q = o(g(Q)^{1-\epsilon})$ for any $\epsilon>0.$

\begin{theorem}\label{main_theorem_all_primes}
Let $g(Q)$ be any increasing function satisfying $g(Q)^{1-\epsilon} \gg Q^{1/2}\log Q$ for any $\epsilon>0$ and $q(Q) \leq Q, \forall Q$. For any prime $q$ let $h_q \ra \infty$ and $\frac{\log h_q}{\log (g(q))} \ra 0$ as $q \ra \infty$. Then with $\MM_{g, q,h_q}(\lambda)$ defined as before, we have 
    \begin{equation}
        \MM_{g,q,h_q}(\lambda) \ra \frac{1}{\sqrt{2\pi}}\int_{-\infty}^{\lambda}e^{-t^2/2}dt, \text{ as }q \ra \infty,
    \end{equation} for any fixed $\lambda$. 
\end{theorem}

\subsection{Discussion on further improvements}
Our paper establishes the central limit theorem for character sums with starting point chosen uniformly at random from $[g(q)]$, for all the primes with $q^{1/2+\epsilon} \ll g(q) \ll q$ and for almost all the primes for practically arbitrary divergent $g(q)\ll q^{\epsilon}$, for all $\epsilon>0$. We believe that our method would give similar kind of improvements, and at least a $\log q$ factor saving in another work of Basak, Nath and Zaharescu \cite{BasakNathZaharescu2}. When it comes to results on all the primes, it remains an open problem to show that a central limit theorem holds with $g(q) = O(q^{1/2-c})$ for some fixed $c>0$. 

\section{Ingredients}
In this section we mention two results, which are necessary for us to establish Theorem \ref{New_Theorem} and Corollary \ref{New_Corollary}. The first ingredient of our proof is based on the following result of Evertse and Silverman \cite{Evertse_Silverman_1986}: 
\begin{fact}{(Evertse-Silverman, 1986)}\\
Set the following notation:
\begin{itemize}
    \item $K$ is an algebraic number field of degree $m$.
    \item $S$ is a finite set of places of $K$, containing the infinite places.
    \item $s = \#S$.
    \item $R_S$ is the ring of $S$- integers of $K$.
    \item $f(X) \in R_S[X]$ is a polynomial of degree $d$ with discriminant $\text{disc}(f) \in R^{*}_S.$
    \item $L/K$ is an extension of degree $M$. 
    \item $\kappa_n(L)$ is the $n$-rank of the ideal class group of $L$.
\end{itemize}For $n \geq 2$, let 
\begin{equation*}
    V(R_S,f,n) = \{x \in R_S: f(x)\in K^{*n}\},
\end{equation*}where $K^{*n}$ denotes the elements of $K$ which are perfect $n$'th powers. 
\begin{enumerate}[label = (\alph*)]
    \item Let $n \geq 3, d \geq 2$ and assume that $L$ contains at least two zeros of $f$. Then 
    \begin{equation*}
        \#V(R_S,f,n) \leq 17^{M(6m+s)}n^{2Ms+\kappa_n(L)}.
    \end{equation*}
    \item Let $d \geq 3,$ and assume that $L$ contains at least three zeros of $f$. Then, 
    \begin{equation*}
        \#V(R_S,f,n) \leq 7^{M(4m+9s)}4^{\kappa_2(L)}.
    \end{equation*}
\end{enumerate}
\end{fact}
As noted by the authors of \cite{BasakNathZaharescu}, this implies the following Corollary, which allows us to put uniform bounds on the number of solutions to equations that we will be working with further in the paper.
\begin{corollary}\label{evertse_corollary}
    Let 
    \begin{equation*}
        f(X) = X(X+\alpha_1)...(X+\alpha_k),
    \end{equation*}where $0 < \alpha_1<...<\alpha_k$ are integers and $k \geq 3$. Then the number of integers solutions to the equation $Y^2 = f(X)$ is $\leq 7^{13+9\omega(D(f))},$ where $\omega(D(f))$ denotes the number of distinct primes dividing the discriminant $D(f)$. 
\end{corollary}
The second ingredient, needed for the proof of the second part of Corollary \ref{New_Corollary} is the strongest currently known statement establishing the order of primes in short intervals due to Baker, Harman and Pintz \cite{Baker_Harman_Pintz}.

\begin{fact}{(Baker-Harman-Pintz, 2001)}\label{baker_harman_pintz}\\
Let $Q$ be sufficiently large. Then for all $\eta \in [0.525,1]$ we have 
\begin{equation}
    \pi(Q+Q^{\eta}) - \pi(Q) \gg \frac{Q^{\eta}}{\log Q}.
\end{equation}
\end{fact}

\section{Improvements}

We start from the following version of \cite[Number Theory Fact~3]{harper2022notecharactersumsshort}, which is a main component separating our approach from the one in \cite{BasakNathZaharescu}. As in \cite{harper2022notecharactersumsshort}, \cite{Davenport2022TheDO} and \cite{BasakNathZaharescu}, we only need upper bounds for the \say{variance} terms that will arise in next section, instead of asymptotic equalities. We get them by employing the Selberg's sieve weights. This allows us to salvage the factor of $\log Q$ that was lost in \cite{BasakNathZaharescu}. Intuitively, the fact below is telling us, that when averaged over a set of prime moduli, the characters behave like the extended Rademacher random multiplicative functions.

\begin{fact}\label{nt_fact}
    Let $f(n)$ be an extended Rademacher random multiplicative function, $\eta \in [0.01,1]$. Then, uniformly for any large $Q, N \leq Q^{\eta}$ and any complex coefficients $(a_n)_{n \leq N}$ we have 
    $$\frac{\log Q}{Q^{\eta}}\sum_{Q \leq q \leq Q+Q^{\eta}}\left|\sum_{n \leq N}a_n\lp(\frac{n}{q}\rp)\right|^2 \ll \E \left|\sum_{n \leq N}a_nf(n)\right|^2 + \frac{1}{Q^{\eta/2}}\left(\sum_{n \leq N}|a_n|\sqrt{s(n)}\right)^2,$$ where the sum on the left hand side is over primes $q$ and $s(n)$ denotes the squarefree part of $n$. 
\end{fact}
\begin{proof}
    The proof is primarily an adaptation of the proof \cite[~ Number Theory Fact 3]{harper2022notecharactersumsshort}, restricting it to short intervals. The same argument appeared earlier in \cite[Lemma 9]{Montgomery_Vaughan_1979}. The main idea is to use the sieve theoretic machinery to upper bound the left hand side, in a way where we obtain some non-trivial cancellations by switching from sum over primes in short intervals, to sum over essentially all numbers. This allows us to use P\'olya-Vinogradov, while keeping contributions coming from perfect squares under control. Note that for prime $q$ and $n \leq q$, we have $\lp(\tfrac{n}{q}\rp) = \lp(\tfrac{s(n)}{q}\rp)$. Hence, we can rewrite the inner sum on the left hand side as 
    \begin{equation*}
        \sum_{\underset{s \text{ squarefree}}{s \leq N} }\lp(\frac{s}{q}\rp)\sum_{\underset{n \leq N}{s(n)=s}}a_n.
    \end{equation*}At the same time, note that from the definition of Rademacher random multiplicative function it follows that 
    \begin{equation*}
        \E\lp|\sum_{ n \leq N}a_nf(n)\rp|^2 = \E\lp|\sum_{\underset{s \text{ squarefree}}{s \leq N} }f(s)\sum_{\underset{n \leq N}{s(n)=s}}a_n\rp|^2 = \sum_{\underset{s \text{ squarefree}}{s \leq N} }\lp|\sum_{\underset{n \leq N}{s(n)=s}}a_n\rp|^2.
    \end{equation*}Now, rewriting the left hand side of the original sum we have 
    \begin{equation}\label{eqnsieve1}
        \frac{\log Q}{Q^{\eta}}\sum_{Q \leq q \leq Q+Q^{\eta}}\lp|\sum_{n \leq N} a_n\lp(\frac{n}{q}\rp)\rp|^2=\frac{\log Q}{Q^{\eta}}\sum_{Q \leq q \leq Q+Q^{\eta}}\sum_{\underset{s \text{ squarefree}}{s_1,s_2 \leq N} }\lp(\frac{s_1s_2}{q}\rp)\left(\sum_{\underset{n \leq N}{s(n)=s_1}}a_n\right)\overline{\left(\sum_{\underset{n \leq N}{s(n)=s_2}}a_n\right)}.
    \end{equation}We proceed by introduction of similar Selberg's sieve weights as \cite{harper2022notecharactersumsshort}, with slight difference coming from the fact that our short intervals can be a lot shorter than the ones there. We can make a choice of a sequence $\lambda_e$ such that $\sum_{e\mid q}\lambda_e \geq \I\{p \mid q, \implies p\geq Q^{\eta/4}\}$ for all q, $\lambda_e=0$, for all $e>Q^{\eta/2}$, $\sum_{Q < q \leq Q+Q^{\eta}}\sum_{e \mid q}\lambda_e \ll \frac{Q^{\eta}}{\log Q} \text{ and }\sum_e|\lambda_e| \ll \frac{Q^{\eta/2}}{(\log Q)^2}$. Using them, we upper bound \eqref{eqnsieve1} by
    \begin{equation}
        \frac{\log Q}{Q^{\eta}}\sum_{\substack{Q\leq q \leq Q+Q^{\eta} \\q \text{ odd }}}\lp(\sum_{e \mid q}\lambda_e\rp)\sum_{\substack{s_1, s_2 \leq N \\ \text{squarefree}}}\lp(\sum_{\substack{n \leq N\\ s(n)=s_1}}a_n\rp)\overline{\lp(\sum_{\substack{n \leq N\\ s(n)=s_2}}a_n\rp)}\lp(\frac{s_1s_2}{q}\rp).
    \end{equation}The main contribution in this kind of sums comes from the terms where $s_1s_2$ (and so $n_1n_2$) is a perfect square. Since $s_1, s_2$ are squarefree, this is counted by the diagonal summands $s_1=s_2$. Notice that we can bound these by the expectation of the sum over the random multiplicative functions and the sieve weights allow us to preserve essentially the density of the primes. 
    \begin{equation*}
        \sum_{\substack{s \leq N\\ \text{squarefree}}}\lp|\sum_{\substack{n \leq N\\s(n)=s}}a_n\rp|^2 \frac{\log Q}{Q^{\eta}}\sum_{\substack{Q \leq q \leq Q+Q^{\eta}\\ q \text{ odd}}}\lp(\sum_{e \mid q}\lambda_e\rp)\lp(\frac{s^2}{q}\rp) \ll \sum_{\substack{s \leq N\\ \text{squarefree}}}\lp|\sum_{\substack{n \leq N \\ s(n)=s}}a_n\rp|^2 = \E\lp|\sum_{n \leq N }a_nf(n)\rp|^2.
    \end{equation*} This step could be done without the help of sieve-theory, but it is necessary in bounding the contributions for $s_1\neq s_2$. In that case, $s_1s_2$ is not a perfect square and we want to use the cancellations coming from the fact that  $(\tfrac{s_1s_2}{q})$ is a non-principal Dirichlet character of a conductor at most $4s_1s_2$. To get these cancellations, sieve theory allows us to pass from sums over primes to sums over essentially all integers in a given interval, at which point we can use the P\'olya-Vinogradov bound. Hence, the off-diagonal contributions can be bounded by 
    \begin{align*}
        &\ll\frac{\log Q}{Q^{\eta}}\sum_{\substack{e \leq Q^{\eta/2}\\e \text{ odd}}}|\lambda_e|\sum_{\substack{s_1\neq s_1 \leq N\\ \text{squarefree}}}\lp|\sum_{\substack{n \leq N\\ s(n)=s_1}}a_n\rp|\lp|\sum_{\substack{n \leq N \\ s(n)= s_2}}a_n\rp|\lp|\sum_{\substack{Q \leq q \leq Q+Q^{\eta}\\ q \text{ odd }\\ e\mid q}}\lp(\frac{s_1s_2}{q}\rp)\rp|
        \\ &\ll \frac{\log Q}{Q^{\eta}}\sum_{\substack{e \leq Q^{\eta/2}\\ e \text{ odd}}}|\lambda_e|\sum_{\substack{s_1 \neq s_2 \leq N\\ \text{ squarefree}}}\lp|\sum_{\substack{n \leq N\\ s(n) = s_1}}a_n\rp| \lp|\sum_{\substack{n \leq N\\ s(n) = s_2}}a_n\rp| \sqrt{s_1s_2}\log(s_1s_2).
    \end{align*} Since our weights satisfy $\sum_e |\lambda_e| \ll \tfrac{Q^{\eta/2}}{(\log Q)^2}$ we can check that this expression is suitably small. 
\end{proof}
\begin{remark}
    Note that a more suitable choice of sieve weights would allow us to save a higher power of $Q$ in the error term, but this will be enough for our goals.
\end{remark}Notice that Number Theory Fact \ref{nt_fact} immediately gives us the following bound 
\begin{equation}\label{bound_on_variance}
    \frac{\log Q}{Q^{\eta}}\sum_{\substack{Q \leq q \leq Q+Q^{\eta}}}\lp|\sum_{n \leq N}a_n\lp(\frac{n}{q}\rp)\rp|^2 \ll \E\lp|\sum_{n \leq N}a_nf(n)\rp|^2 + \frac{N}{Q^{\eta/2}}\lp(\sum_{n \leq N}|a_n|\rp)^2. 
\end{equation}This allows us to give improvements to \cite[Lemma~3.2]{BasakNathZaharescu}, which we rewrite in terms of an average of the primes in the interval $Q, Q+Q^{\eta}$. The point, is that the right hand side would then tend to zero. As we will see in the proof of Theorem \ref{New_Theorem}, it is enough to prove this and subsequent results for $\eta=0.01$, which we have chosen as our lower bound on $\eta$. 
\begin{lemma}\label{variance_lemma_prob}
    Let $r$ be a fixed positive integer. Let $g$ be an increasing, divergent function such that 
    \begin{itemize}
        \item $g'(c)=O\lp((g(Q))^{0.99}Q^{-0.01}\rp), \forall c \in [Q,Q+Q^{0.01}]$ as $Q \ra \infty$,
        \item $g(Q) \ll Q^{\delta}, \forall \delta>0$.
    \end{itemize}Let $r \leq h \leq (g(Q+Q^{0.01}))^{\frac{1}{2500r^2}}.$ Then we have 
    \begin{equation}
        \frac{\log Q}{Q^{0.01}}\sum_{\substack{Q \leq q \leq Q+Q^{0.01} \\ q \text{ prime }}}\lp(\frac{1}{g(Q)}\sum_{1 \leq m \leq g(Q)}S_h^{2r}(m,q) - \mu_{2r}(h-\theta r)^r\rp)^2 \ll (g(Q))^{-1+1/300}
    \end{equation}and 
    \begin{equation}
       \frac{\log Q}{Q^{0.01}} \sum_{\substack{Q \leq q \leq Q+Q^{0.01} \\ q \text{ prime }}}\lp(\frac{1}{g(Q)}\sum_{1 \leq m \leq g(Q)}S_h^{2r-1}(m,q)\rp)^2 \ll (g(Q))^{-1+1/300},
    \end{equation}where $\mu_{2r}$ are the $2r'$th moments of the standard normal and $\theta : = \theta(h,r) \in [0,1]$ is defined in the proof. 
\end{lemma}
\begin{proof}We shall only prove the lemma for $2r,$ the second statement follows the same argument. 
    Our proof uses the same ides as the proof in \cite{BasakNathZaharescu}, but first passing from the relevant sum to expectations. Note that the term $-\mu_{2r}(h-\theta r)^r$ can be taken to be the coefficient of $\lp(\tfrac{1}{q}\rp)$. Using  \eqref{bound_on_variance} for $\eta=0.01$ we have, for $f$ being the extended Rademacher random multiplicative function
    \begin{align*}
        &\hspace{25pt}\frac{\log Q}{Q^{0.01}}\sum_{\substack{Q \leq q \leq Q+Q^{0.01} \\ q \text{ prime }}}\lp(\frac{1}{g(Q)}\sum_{1 \leq m \leq g(Q)}S_h^{2r}(m,q) - \mu_{2r}(h-\theta r)^r\rp)^2\\
        &\ll \E\lp[\lp|\frac{1}{g(Q)}\sum_{1 \leq m \leq g(Q)}\sum_{\alpha \in [h]^{2r}}f\lp(\prod_{i=1}^{2r}m+\alpha_i\rp)- \mu_{2r}(h-\theta r)^r\rp|^2\rp] + \frac{(g(Q)+h)^{2r}}{Q^{0.005}}\lp(\sum_{n \leq (g(Q)+h)^{2r}} |a_n|\rp)^2,
    \end{align*}with constants $a_n$ implied by construction. We will want the right hand side of the above to go to zero, as that would imply further that, as $Q \ra \infty,$ we will be able to prove the gaussian behaviour for almost all the primes in the interval, using method of moments argument. First, we deal with bounding the expectation term. It is equal to 
    \begin{align}\label{big_expectation}
        &\frac{1}{(g(Q))^{2}}\E \lp[\sum_{1 \leq m_1,m_2 \leq g(Q)}\sum_{\alpha,\beta \in [h]^{2r}}f\lp(\prod_{i=1}^{2r}(m_1+\alpha_i)(m_2+\beta_i)\rp)\rp] \\
        &-2\mu_{2r}(h-\theta r)^r\E\lp[\frac{1}{g(Q)}\sum_{1 \leq m \leq g(Q)}\sum_{\alpha \in [h]^{2r}}f\lp(\prod_{i=1}^{2r}(m+\alpha_i)\rp)\rp]+\mu_{2r}^2(h-\theta r)^{2r}.
     \end{align}Starting now, we proceed to conduct the same analysis as \cite{BasakNathZaharescu}. Note however, that unlike Dirichlet characters, random multiplicative functions are perfectly orthogonal, which makes the analysis much easier. We start with the second expectation and proceed accordingly. In both cases, we will show that as in the case of Dirichlet characters, the most significant contributions to the sum will come from these $n$'s which are perfect squares. Hence, we fix $\alpha \in [h]^{2r}$ and estimate for how many $m's$ in our range $\prod_{i=1}^{2r}(m+\alpha_i)$ is a perfect square. For this purpose consider the following algorithm: 
\begin{algorithm}
 Consider $\alpha_1$, the first coordinate of $\alpha$. Then, in order, go through the remaining coordinates of $\alpha$ until we find $\alpha_j$ such that $\alpha_j = \alpha_i$.  If no such $\alpha_j$ exists, proceed to do the same with $\alpha_2$ and so on, until we go through all coordinates of the vector. If such coordinate exists proceed to delete the two coordinates from the vector, and repeat the procedure for the next non-deleted coordinate of $\alpha$. The algorithm terminates when we reach the last non-deleted coordinate of $\alpha$. Let us call the algorithm $\Psi$ and the resulting new vector $\gamma=\Psi(\alpha)\in [h]^{2k}, k \leq r$. 
\end{algorithm}
Notice that for any $m, \prod_{i=1}^{2r}(m+\alpha_i)$ is a perfect square iff $\prod_{i=1}^{2k}(m+\gamma_i)$ is one. Consider the three following cases:
\begin{itemize} 
    \item Assume $k \geq 2$. We want to calculate the number of $m \in \{1,...,g(Q)\}$ such that $\prod_{i=1}^{2k}(m+\gamma_i)$ is a perfect square. Let $f(m) = \prod_{i=1}^{2k}(m+\gamma_i)$. Corollary \ref{evertse_corollary} says that given distinct $\gamma_i$ (as in our case), the amount of solutions to the diophantine equation 
\begin{equation}\label{evertse}
Y^2 = f(m)
\end{equation} is less than $7^{13+9\omega(D(f))}$, where $\omega$ is the distinct prime divisor function and $D(f)$ is the discriminant of the polynomial $f$. We have the well known bound (see for example \cite{Montgomery_Vaughan_2006}) 
\begin{equation*}
\omega(n) \leq \frac{\log n}{\log \log n}(1+o(1))
\end{equation*}Note that the RHS of the inequality is an increasing function of $n$. The discriminant of $f(m)$ in our case is given by $\prod_{i<j}^{2k}(\gamma_i-\gamma_j)^2$. Therefore, we can say that the amount of solutions to \eqref{evertse} is bounded by 

\begin{align*}
&\ll e^{9\log 7\omega(D(f))} \ll \lp(e^{(\log D(f))}\rp)^{9\log 7/\log \log D(f)} \\
&= \lp(\prod_{i<j}^{2k}(\gamma_i-\gamma_j)^2\rp)^{9 \log 7/ \log \log (\prod_{i<j}^{2k}(\gamma_i-\gamma_j)^2)} \ll \lp(h^{4r^2}\rp)^{9\log 7 / \log \log ( h^{4r^2})}\\
&\ll_{\epsilon} (g(Q))^{\epsilon},\\
\end{align*}which we get recalling $h \leq g(Q+Q^{0.01})^{1/2500r^2}$, for any $\epsilon >0$.
   \item Assume $k=1$. In this case we want to count the number of solutions to the diophantine equation
\begin{equation*}
Y^2 = (m+\alpha_i)(m+\alpha_j),
\end{equation*}and we use an elementary divisor function argument to find a bound of sufficient order. Without loss of generality assume $\alpha_i<\alpha_j$. Let $D = m+\alpha_i, Z = \alpha_j-\alpha_i$. Then, we want to find to find the number of integer solutions to 
\begin{equation*}
Y^2 = D(D+Z) \iff 4Y^2+Z^2 = (2D+Z)^2 \iff Z^2 = (2D+Z-2Y)(2D+Z+2Y).
\end{equation*} Therefore, notice that given $Z$, $D, Y$ must be such that $2D+Z-2Y = d, 2D+Z+2Y = Z^2/d$, for some divisor $d$ of $Z^2$. Since for any such divisor, the above gives a unique solution, the number of such pairs is trivially bounded by $\tau(Z^2)\ll h^{\epsilon} \ll (g(Q))^{ \epsilon}$, for any $\epsilon>0.$
     \item Assume $k=0$. In this case, $\prod_{i=1}^{2r}(m+\alpha_i)$ is a perfect square for all $m$. 
\end{itemize}Write $K(r, h)$ for the number of vectors arising in case $3$. 
Now, notice that if $n$ is not a perfect square we have $$\E [a_nf(n)] = a_n\E[f(n)] = a_n\E[f(s)]= 0,$$ where $s$ is the squarefree part of $n$. Therefore we have
\begin{equation}\label{expectation_1}
      \E\lp[\frac{1}{g(Q)}\sum_{1 \leq m \leq g(Q)}\sum_{\alpha \in [h]^{2r}}f\lp(\prod_{i=1}^{2r}m+\alpha_i\rp)\rp] = K(r,h) + O(g(Q)^{-1+\epsilon}g(Q)^{2r/2500r^2}),
\end{equation} for any $\epsilon >0.$ Therefore, we need to estimate $K(r,h)$. This is the amount of $\alpha \in [h]^{2r}$ such that $\Psi(\alpha)$ vanishes. Notice that there are ${h \choose r}$ ways to choose exactly $r$ numbers less than or equal to $h$ without repetitions. Out of these numbers, we have $r$ options as the choice for the first coordinate of $\alpha$ and $2r-1$ choices for the coordinate where it will appear again. Proceeding that way says that there are at least 
\begin{equation*}
    {h \choose r}\cdot r \cdot (2r-1)\cdot (r-1) \cdot (2r-3)\cdot ... \cdot 2 \cdot 3\cdot 1={h \choose r}\cdot r!\cdot (2r-1)\cdot(2r-3)...\cdot 1 = h\cdot (h-1)\cdot...\cdot (h-r+1) \cdot \mu_{2r}
\end{equation*}ways of choosing $\alpha$ such that $\Psi(\alpha)$ vanishes, so $K(r,h) \geq \mu_{2r}h(h-1)...(h-r+1)$.\newline
At the same time, note that if for that to happen, we need to have at most r distinct pairs of numbers from $\{1,2,...,h\}$ to form $\alpha$. The number of ways of choosing them is $h^r$ and after the choice, there is at most $(2r-1) \cdot (2r-3) \cdot ...\cdot 1$ ways of arranging them as $r$ pairs. Therefore, we can see that $K(r,h) \leq \mu_{2r}h^r$.  Hence we have, that for any $h,r$ there exists $\theta: = \theta(h,r)\in [0,1] \text{ so that}$
\begin{equation}\label{definition_of_theta}
    K(r,h) = \mu_{2r}(h-\theta r)^r.
\end{equation}
The proof proceeds similarly when estimating the first expectation, but we need to modify the algorithm slightly as well as deal with some differences which arise in cases $1$ and $3$ of the algorithm. For convenience, we shall use the same notation. 
We shall now deal with the first expectation term in a similar way, so we want to count the contribution of the perfect squares $\prod_{i=1}^{2r}(m_1+\alpha_i)(m_2+\beta_i)$. As before, we do it through counting $m_1$ which form perfect squares for fixed vectors $\alpha, \beta \in [h]^{2r}$ and fixed $m_2$. 
Consider the following algorithm:
\begin{algorithm}\label{algorithm_2}
     Let $m:= m_2-m_1$ and let  $\xi = (m+\beta_1, m+\beta_2,...,m+\beta_{2r}).$ Then, adjoin the vectors $\alpha, \xi$ to create the vector $(\alpha_1,...,\alpha_{2r}, \xi_1, ..., \xi_{2r})\in [m+h]^{4r}$. Then, proceed with the same argument and map $\Psi$ as before, but this time the vector is of length $4r$, and the output is vector $\Psi((\alpha, \xi)) =:\gamma \in [m+h]^{2k},  k \leq 2r$.  Notice that $\prod_{i=1}^{2r}(m_1+\alpha_i)(m_2+\beta_i)$ is a square iff $\prod_{i=1}^{2k}(m_1+\gamma_i)$ is. 
\end{algorithm}As before, we follow with case by case analysis. 
\begin{itemize}
    \item Assume $k \geq 2$. Then, again by Corollary \ref{evertse_corollary} we can bound the number of solutions to the diophantine equation 
    \begin{equation*} 
     Y^2 = f(m_1) = (m_1+\gamma_1)...(m_1+\gamma_{2k}).
    \end{equation*}As above, we use Number Theory Fact \ref{evertse_corollary} to bound the number of these solutions by 
    \begin{align*}
     &\ll e^{9\log 7\omega(D(f))} \ll \lp(e^{(\log D(f))}\rp)^{9\log 7/\log \log D(f)} \\   
     &\ll \lp((g(Q)+h)^{8r^2}h^{8r^2}\rp)^{9\log 7/\log \log \lp((g(Q)+h)^{8r^2}h^{8r^2}\rp)} \\
     &\ll_{\epsilon, r} (g(Q))^{ \epsilon},
    \end{align*} for any $\epsilon>0$.
    \item Assume $k=1$. This works exactly as before, so amount of solutions $m_1$ to this equation can also be bounded by $(g(Q))^{\epsilon}$. 
    \item Assume $k=0$. In this case, note that if $|m| \geq h+1$ then $\prod_{i=1}^{2r}(m_1+\alpha_i)$ and $\prod_{i=1}^{2r}(m_2+\beta_i)$ both are perfect squares, each with at most $r$ distinct coefficients (since $k=0$). Therefore, the number of vectors $\alpha, \beta$ for which both products are perfect squares for all pairs of $m_1, m_2$ is exactly $\mu_{2r}^2(h-\theta r)^{2r}$. If $|m| \leq h$ then even if we assume that there is a solution for any vectors $\alpha,\beta$ and any $m_2$, but for each such choice there is at most $O(h)$ choices for $m_1$. 
\end{itemize}
Now, notice similarly to before, by orthogonality of Rademacher random multiplicative functions that the expectation for non-square products will have null contribution. Using \eqref{expectation_1} and summing over all possible $\alpha, \beta$ for each case discussed after Algorithm \ref{algorithm_2} implies that we can bound the original expectation \eqref{big_expectation} by
\begin{align*}
     \mu_{2r}^2(h-\theta r)^{2r} &+O\lp((g(Q))^{-1+\epsilon + 4(r^2+1)/2500r^2}\rp) \\
     &- 2\mu_{2r}^2(h-\theta r)^{2r} + O\lp((g(Q)^{-1+\epsilon+3r^2/2500r^2}\rp) \\
     &+ \mu_{2r}^2(h-\theta r)^{2r} \ll (g(Q))^{-1+1/300}.
\end{align*} This concludes dealing with bounds on expectations. What is left now is to complete the proof of our Lemma is to show that the term

\begin{equation}
    \frac{(g(Q)+h)^{2r}}{Q^{0.005}}\lp(\sum_{n \leq (g(Q)+h)^{2r}}|a_n|\rp)^2
\end{equation} goes to zero at least as quickly. Notice that for $n>1, |a_n|$ can be explicitly described as the count of representations of $n$ as a product $\prod_{i=1}^{2r}(m+\alpha_i)$ for $ 1 \leq m \leq g(Q),  \alpha \in [h]^{2r}$, later normalised by $\frac{1}{g(Q)}$. Since every choice of $\alpha$ and $m$ represents some $n$, their sum (still averaged over all $m$) is easily seen to be 
\begin{equation*}
     \sum_{1<n \leq (g(Q)+h)^{2r}}|a_n| = h^{2r}.
\end{equation*}Additionally, we pick up the $-\mu_{2r}(h-\theta r)^r$ term for $n=1$, so that 
\begin{equation*}
     \frac{(g(Q)+h)^{2r}}{Q^{0.005}}\lp(\sum_{n \leq (g(Q)+h)^{2r}}|a_n|\rp)^2 \ll  \frac{(g(Q)+h)^{2r} (g(Q))^{2r/(2500r^2)}}{Q^{0.005}}\ll \frac{(g(Q))^{2r(1+1/(2500r^2))}}{Q^{0.005}},
\end{equation*}which is of sufficiently small order. This finishes our argument for $2r$. Notice that in case $2r-1,$ the reduction algorithms never give $k=0$ as $2r-1$ is odd and so the term with $\mu$ doesn't arise, which proves the lemma fully. 
\end{proof}

This implies the following Corollary, which can be viewed as an equivalent to Chebyshev's inequality.

\begin{corollary}\label{Chebyshev}
   Let $r$ be a fixed positive integer. Let $g$ be an increasing, divergent function such that 
    \begin{itemize}
        \item $g'(c)=O\lp((g(Q))^{0.99}Q^{-0.01}\rp), \forall c \in [Q,Q+Q^{0.01}]$ as $Q \ra \infty$,
        \item $g(Q) \ll Q^{\delta}, \forall \delta>0$.
    \end{itemize}For each prime $q \in [Q, Q+Q^{0.01}]$ choose $h_q$ such that  
    \begin{equation*}
        r \leq h_q \leq g(q)^{\frac{1}{2500r^2}}.
    \end{equation*}Define
    \begin{equation*}
        \EE_1(Q; g, r): = \lp\{q \in \PP\cap[Q,Q+Q^{0.01}]: \lp|g(Q)^{-1}\sum_{m \leq g(Q)}S_{h_q}^{2r}(m,q)-\mu_{2r}(h_q-\theta_qr)^r\rp| \geq \frac{1}{g(q)^{1/8}}\rp\}
    \end{equation*}and 
    \begin{equation*}
        \EE_2(Q; g, r): = \lp\{q \in \PP\cap[Q,Q+Q^{0.01}]: \lp|g(Q)^{-1}\sum_{m \leq g(Q)}S_{h_q}^{2r-1}(m,q)\rp| \geq \frac{1}{g(q)^{1/8}}\rp\}.
    \end{equation*}Then 
    \begin{equation*}
        \lp|\EE_1 \cup \EE_2\rp| \ll \frac{Q^{0.01}}{(\log Q)g(Q)^{74/100}}.
    \end{equation*} 
\end{corollary}

\begin{proof}
    We have 
    \begin{align*}
        &\sum_{q \in \EE_1}g(q)^{1/4}\lp(\frac{1}{g(Q)}\sum_{1 \leq m \leq g(Q)}S_{h_q}^{2r}(m,q) - \mu_{2r}(h_q-\theta_qr)^r\rp)^2\\
        &\ll g(Q)^{1/4} \sum_{\substack{Q \leq q \leq Q+Q^{0.01}\\ q \text{ prime}}}\lp(\frac{1}{g(Q)}\sum_{1 \leq m \leq g(Q)}S_{h_q}^{2r}(m,q) - \mu_{2r}(h_q-\theta_qr)^r\rp)^2.\\ 
       & =g(Q)^{1/4}\sum_{h \leq g(Q+Q^{0.01})^{1/(2500r^2)}}\sum_{\substack{Q \leq q \leq Q+Q^{0.01} \\ q \text{ prime }\\ h_q=h}}\lp(\frac{1}{g(Q)}\sum_{1 \leq m \leq g(Q)}S_{h}^{2r}(m,q) - \mu_{2r}(h-\theta r)^r\rp)^2\\
       &\ll g(Q)^{1/4} g(Q)^{1/(2500r^2)}\sum_{\substack{Q \leq q \leq Q+Q^{0.01} \\ q \text{ prime }}}\lp(\frac{1}{g(Q)}\sum_{1 \leq m \leq g(Q)}S_{h}^{2r}(m,q) - \mu_{2r}(h-\theta r)^r\rp)^2\\
       &\ll \frac{Q^{0.01}}{\log Q g(Q)^{74/100}},
    \end{align*}where the last line follows in an obvious way from Lemma \ref{variance_lemma_prob}. At the same time, 
    \begin{align*}
            &\sum_{q \in \EE_1}g(q)^{1/4}\lp(\frac{1}{g(q)}\sum_{1 \leq m \leq g(q)}S_h^{2r}(m,q) - \mu_{2r}(h_q-\theta_qr)^r\rp)^2 \gg \lp|\EE_1\rp|,
    \end{align*}which implies
    \begin{equation*}
        \lp|\EE_1\rp| \ll \frac{Q^{\eta}}{(\log Q)g(Q)^{74/100}}.
    \end{equation*} The same result for $\EE_2$ implies the Corollary. 
\end{proof}

\section{Proof of Theorem \ref{New_Theorem} and Corollary \ref{New_Corollary}}
In this section we put it all together in a method of moments argument to prove Theorem \ref{New_Theorem}. 
\begin{proof}
    Consider $\eta \in [0.01,1]$. We begin with an observation, that if a subset $\BB \subseteq \PP$ of primes is $\eta_1$-sufficient, it is also $\eta_2$-sufficient for any $\eta_1<\eta_2 \leq 1$. One can prove it with a simple covering argument of the set $[X,X+X^{\eta_2}]$ with shorter intervals of form $[Q,Q+Q^{\eta_1}]$. Therefore, it is enough for our purposes to prove the theorem when $\eta=0.01$. For each $k \in \N$ set $R_k = 2^k$. By Corollary \ref{Chebyshev} there exists an increasing sequence of integers $Q_k$ such that for each $k$, for all $Q \geq Q_k$ and all $r \leq R_k$ 
    \begin{equation}\label{good_primes_1}
        \lp|\frac{1}{g(Q)}\sum_{m \leq g(Q)}S_{h_q}^{2r}(m,q)-\mu_{2r}(h_q-\theta_q r)^r\rp|< \frac{1}{g(q)^{1/8}}
    \end{equation}and 
    \begin{equation}\label{good_primes_2}
         \lp|\frac{1}{g(Q)}\sum_{m \leq g(Q)}S_{h_q}^{2r-1}(m,q)\rp|< \frac{1}{g(q)^{1/8}},
    \end{equation}for all but 
    \begin{equation*}
       \leq \frac{Q^{0.01}}{(\log Q)g(Q)^{7/10}}
    \end{equation*} many primes in the interval $[Q, Q+Q^{0.01}].$ To apply the Corollary, we use the assumption that $h_q \ra \infty$ as $q \ra \infty$ of Theorem \ref{New_Theorem}. Notice that we used some factors of $g(Q)$ to get rid of Vinogradov notation in place of an inequality, so with choice $g(Q) = (\log Q)^A, A>0$ we could preserve the property of being $\eta$-strong of \cite{BasakNathZaharescu}. For any $Q_k <Q_{k+1}$ let $\BB_{k}$ be the set of primes $q \in [Q_k, Q_{k+1}]$, such that $q \in [N,N+N^{0.01}]$ for some $N \in [Q_k,Q_{k+1}]$ and \eqref{good_primes_1},\eqref{good_primes_2} hold with $Q=N$ for all $r \leq R_k$. Then for all $q \in \BB_{k},$ and $j \in \N$ we have that there exists $N$ as described before, such that
    \begin{align*}
      &g(q)^{-1}\sum_{m \leq g(q)}\lp((h_q)^{-1/2}S_{h_q}(m,q)\rp)^j \\  
      &= (g(N)+g'(c)(q-N))^{-1} \sum_{m \leq g(N)+g'(c)(q-N)}\lp((h_{q})^{-1/2}S_{h_q}(m,q)\rp)^j,
      \end{align*}for some $c \in (N,N+N^{0.01})$. By Corollary \ref{Chebyshev}, assumptions on $g'$, \eqref{good_primes_1} and \eqref{good_primes_2}, this is equal to
      \begin{align}\label{clean_up_equation}
      &= g(N)^{-1}(1+o(g(N)^{-0.01}))^{-1}\sum_{m \leq g(N)(1+o(g(N)^{-0.01}))}\lp((h_{q})^{-1/2}S_{h_{q}}(m,q)\rp)^j \\
      &=g(N)^{-1}(1+o(1))^{-1}\sum_{m \leq g(N)}\lp((h_{q})^{-1/2}S_{h_{q}}(m,q)\rp)^j + o(1) \underset{}{\ra} \mu_j,
      \end{align}as $k \ra \infty$, where $\mu_j$ is the $j$'th moment of a standard Gaussian random variable. Define 
    \begin{align*}
        \BB &:= \cup_{k=1}^{\infty}\BB_k.
    \end{align*}We can immediately see that by construction, $\BB$ is an $0.01$-sufficient set of prime, so also $\eta$-sufficient for all $\eta \in[0.01,1]$. Consider a sequence $(q_n) \in \BB$. Then, notice that by \eqref{clean_up_equation} we have for any fixed positive $j$
    \begin{equation}\label{convergence_to_moments}
        g(q_n)^{-1}\sum_{m \leq g(q_n)}\lp((h_{q_n})^{-1/2}S_{h_{q_n}}(m,q_n)\rp)^j  \ra \mu_j,
    \end{equation}as $n \ra \infty$. For convenience of presentation define
    \begin{equation*}
        \NN_{q}(s) : = \lp|\lp\{1 \leq m \leq g(q): S_{h_q}(m,q) \leq s\rp\}\rp|. 
    \end{equation*}This is a non-decreasing function of $s$, constant except for discontinuities at certain integral values of $s$. Also 
    \begin{equation*}
      \NN_q(s) =   \left\{ \begin{array}{rcl}
 0 & \mbox{for}
 & s< -h_q \\ \lfloor g(q)\rfloor & \mbox{for} & s > h_q.
 \end{array}\right.
    \end{equation*}We can recognise
    \begin{equation*}
        \MM_{g,q,h_q}(\lambda) = \frac{1}{g(q)}\NN(\lambda h_q^{1/2}).
    \end{equation*} Note that for any $q,h_q$ this defines a natural cumulative distribution function $\Psi_q$ on the real line, equipped with a probability measure. The goal now, is to show that as $n \ra \infty$ for $(q_n) \in \BB$, this distribution converges weakly to a Gaussian. 
    Now, rewriting \eqref{convergence_to_moments} with the new notation gives 
    \begin{equation*}
        g(q_n)^{-1}\sum_{s=-h_{q_n}}^{h_{q_n}}(h_{q_n}^{-1/2}s)^j\lp(\NN_{q_n}(s)-\NN_{q_n}(s-1)\rp) = \int_{-\infty}^{\infty}t^jd\Psi_{q_n} \ra \mu_j, 
    \end{equation*} as $n \ra \infty$ for any fixed positive $j$.
    This implies, that the random variable $X$ coming from the distribution associated to the law of distribution $\P(X \leq \lambda) = M_{g,q,h_q}(\lambda)$ has identical moments to those of the standard normal. From a well-known probabilistic result, we know that if moments of a sequence of random variables converge to moments of the standard normal, this completely determines their limiting distribution as that of a standard normal. In fact, this is true also in more general setting, where the limiting moments are finite and unique to some probability distribution. Therefore, we have the convergence of cumulative distribution functions 
    \begin{equation}
        \MM_{g, q_n, h_{q_n}}(\lambda) \ra \frac{1}{\sqrt{2\pi}}\int_{-\infty}^{\lambda}e^{-t^2/2}dt,
    \end{equation}which completes the proof of Theorem \ref{New_Theorem}. To prove Corollary  \ref{New_Corollary} simply apply Number Theory Fact \ref{baker_harman_pintz}.
\end{proof}

\section{Results for all the primes}
For the remainder of this section let us forget about previous assumptions and let $g(q)$ be an arbitrary increasing, divergent function. Davenport and Erd\H{o}s \cite{Davenport2022TheDO} and Lamzouri \cite{Lamzouri2011TheDO} establish \eqref{distributional_convergence} for all primes provided $h_q \ra \infty,$ but $h_q = q^{o(1)},$ with the starting points of the sum being uniform in $\{1,...,q\}$, while our method establishes \eqref{distributional_convergence} for almost all the primes, with $h_q \ra \infty, h_q =g(q)^{o(1)}$ and the intervals of size $g(q)$, for any function $g(q)$ as in our results. In fact, if $g(q)$ is small enough (e.g. of size $(\log q)^{1/2}$), then there exist primes $q$ for which $(\tfrac{n}{q})=1$ for all $n \leq g(q),$ which makes it impossible to improve Theorem \ref{New_Theorem} to cover the whole set of primes. This suggests, that there is some threshold for the size of the interval $g(q)$, where \eqref{distributional_convergence} changes from convergence for almost all, to all the primes. This problem takes us back to the reason why we even considered averaging over primes in intervals: To get distributional results for sums as short as ours, there is the need to employ cancellations coming from changes in sign of $(\tfrac{n}{q})$, when $n$ is not a perfect square and $q$ varies. If we consider just one prime $q$, we are forced to understand the structure and get necessarily good cancellations in sums of the form 

\begin{equation}
    \sum_{1 \leq m \leq g(q)}S_{h_q}^{2r}(m) = \sum_{1 \leq m \leq g(q)}\sum_{\alpha_1=1}^{h_q}\cdots \sum_{\alpha_{2r}=1}^{h_q}\lp(\frac{(m+\alpha_1)(m+\alpha_2)...(m+\alpha_{2r})}{q}\rp),
\end{equation}where 
\begin{equation*}
    S_{h_q}(x): = \sum_{x \leq n \leq x+h_q}\lp(\frac{n}{q}\rp).
\end{equation*}If we consider fixed $\alpha \in [h_q]^{2r}$, as in the proof of Lemma \ref{variance_lemma_prob}, we see that the main contribution comes from $\alpha$, which all components appear inside the vector an even amount of times, in which case $f(m): = \prod_{i=1}^{2r}(m+\alpha_i)$ is a square of a polynomial. This is case $3$ of either of the algorithms of that proof. On the other hand, when dealing with the case when this product is not of this form, there is no prime averaging to come to our rescue, but we still need to show that 

\begin{equation*}
    \sum_{1 \leq m \leq g(q)}'S_{h_q}^{2r}(m) = \sum_{1 \leq m \leq g(q)}'\sum_{\alpha_1=1}^{h_q}\cdots \sum_{\alpha_{2r}=1}^{h_q}\lp(\frac{(m+\alpha_1)(m+\alpha_2)...(m+\alpha_{2r})}{q}\rp) = o(g(q)h_q^r),
\end{equation*}where the sums above are over $\alpha \in [h_q]^{2r}$ for which $f(m)$, as defined above is not a square of a polynomial. In the original paper of Davenport and Erd\H{o}s, they use the Weil bound \cite{Weil1948SurLC}, which allows them to show that if $g(q)=q$, the above holds. We have 
\begin{equation*}
    \sum_{1 \leq m \leq q}\lp(\frac{(m+n_1)(m+n_2)...(m+n_{2r})}{q}\rp) = O(q^{1/2}),
\end{equation*}for these $n$'s that do not contribute to case $3$ in any of algorithms in Lemma \ref{variance_lemma_prob}. In this section, we use a fact about incomplete character sums, that follows from the Weil bound \cite{Weil1948SurLC} and show that the result on distribution of sums of Legendre character for all the primes, holds true also if $g(q) \gg (q^{1/2}\log q)^{1/(1-\epsilon)}$, for any $\epsilon>0$. The main ingredient that we use is \cite[Theorem ~ 2]{Maduit_Sarkozy_1997}.

\begin{fact}{(Mauduit - S\'arközy, 1997)}\label{incomplete_sum_lemma}\\
    Let $q$ be a prime number, $\chi$ a non-principal character modulo $q$ of order $d$, and $f(x)\in \F_q[x]$ have degree $k$ and a factorization $$f(x) = b(x-x_1)^{d_1}(x-x_2)^{d_2}...(x-x_s)^{d_s},$$ with $x_i \neq x_j$ for $i\neq j$ in $\overline{\F_q}$. Then, if $(d, d_1, d_2,...,d_s) =1 $, then for any real numbers $X_, 0 < Y \leq q$ we have 
    \begin{equation}
        \lp| \sum_{X < n \leq X+Y } \chi(f(n))\rp| < 9 k q^{1/2}\log q.
    \end{equation}
\end{fact}

The proof of the Theorem \ref{main_theorem_all_primes} is essentially the same method of moments proof as that of \cite{Davenport2022TheDO} and is much simpler than the one presented for almost all the primes before. 
\subsection{Proof of Theorem \ref{main_theorem_all_primes}}

\begin{proof}
    The main idea of the method of moments remains unchanged, except this time we do not allows ourselves to average over different primes. We follow the original argument due to Davenport and Erd\H{o}s \cite{Davenport2022TheDO}. First, for any $r \in \N_{>0}$ and any sufficiently large prime $q$ we want to show that 
    \begin{equation}
        \frac{1}{g(q)}\sum_{m \leq g(q)}S_{h_q}^{2r}(m,q) - \mu_{2r}(h-\theta_2r)^r
    \end{equation} and 
    \begin{equation}
        \frac{1}{g(q)}\sum_{m \leq g(q)}S_{h_q}^{2r-1}(m,q)
    \end{equation}can be bounded sufficiently well in terms of some function of $q$. Using that $h_q \ra \infty$ as $q \ra \infty$ we can assume that $r \leq h$. We first deal with the case when the exponent of the sum is $2r$. Write 
    \begin{equation*}
        \sum_{m \leq g(q)}S_{h_q}^{2r} = \sum_{m \leq g(q)}\sum_{\alpha_1=1}^{h_q}\cdots \sum_{\alpha_{2r}=1}^{h_q}\lp(\frac{(m+\alpha_1)(m+\alpha_2)...(m+\alpha_{2r})}{q}\rp)
    \end{equation*}First, consider only these $\alpha \in [h_q]^{2r}$ for which the product above is a square of a polynomial in $\F_q$. These are exactly the $\alpha$ that constitute the case $3$ of the algorithms described in the proof of Lemma \ref{variance_lemma_prob}. For these $\alpha$, using the argument of Lemma \ref{variance_lemma_prob}, there exists $\theta_1 \in [0,1]$ such that 
    \begin{equation}\label{squares_eqn}
        \sum_{m \leq g(q)}\lp(\frac{(m+\alpha_1)(m+\alpha_2)...(m+\alpha_{2r})}{q}\rp) = g(q)- \theta_1r.
    \end{equation} As in the proof of Lemma \ref{variance_lemma_prob}, we estimate the number of $\alpha \in [h_q]^{2r}$ like this to be $K(r,h_q) = \mu_{2r}(h_q-\theta_2 r)^r$ for some $\theta_2 \in [0,1]$.
    For any other $\alpha$ under consideration, we perform on it the algorithm of Lemma \ref{variance_lemma_prob}. If we define $\gamma \in [h_q]^{2k}$ to be the newly acquired vector with $k \leq r$, the problem becomes that of dealing with estimating the sum 
    \begin{equation}\label{non_squares_all_primes}
        \sum_{m \leq g(q)}\lp(\frac{(m+\gamma_1)(m+\gamma_2)...(m+\gamma_{2k})}{q}\rp).
    \end{equation}If we write
    \begin{equation*}
        f(m) = (m+\gamma_1)(m+\gamma_2)...(m+\gamma_{2k})
    \end{equation*}then \eqref{non_squares_all_primes} is susceptible to the Number Theory Fact \ref{incomplete_sum_lemma}, which gives us 
    \begin{equation}\label{non_squares_eqn}
        \lp|\sum_{m \leq g(q)}\lp(\frac{(m+\alpha_1)(m+\alpha_2)...(m+\alpha_{2r})}{q}\rp)\rp| = \lp|\sum_{m \leq g(q)}\lp(\frac{(m+\gamma_1)(m+\gamma_2)...(m+\gamma_{2k})}{q}\rp)\rp| < 18kq^{1/2}\log q.
    \end{equation} Therefore, combining \eqref{squares_eqn}, \eqref{non_squares_eqn} and the triangle inequality we get 
    \begin{equation}\label{even_exponent_case_all_primes}
        \sum_{m \leq g(q)}S_{h_q}^{2r}(m,q) - \mu_{2r}(h_q-\theta_2r)^rg(q) \ll h_q^{2r}q^{1/2}\log q.
    \end{equation}Similarly, since in the case when the exponent of the sum $S_h(m,q)$ is $2r-1$, the product 
    \begin{equation*}
        f(m)= (m+\alpha_1)(m+\alpha_2)...(m+\alpha_{2r-1})
    \end{equation*} is never a square of a polynomial in $\F_p$, the second part of the above argument shows
    \begin{equation}\label{odd_exponent_case_all_primes}
        \sum_{m \leq g(q)}S_{h_q}^{2r-1}(m,q)  \ll h_q^{2r-1}q^{1/2}\log q.
    \end{equation}
    Therefore, combining \eqref{even_exponent_case_all_primes} and \eqref{odd_exponent_case_all_primes} we have that for all $r \in \N_{>0}$ we have
    \begin{equation}
        \frac{1}{g(q)} \sum_{m \leq g(q)}(h_q^{-1/2}S_h(m,q))^{r} = \mu_r + O(g(q)^{-1}h_q^{r/2}q^{1/2}\log q ),
    \end{equation}where $\mu_r$ is the $r'$th moment of the standard normal, so that as long as $h_q^{r}q^{1/2}\log q = o(g(q))$, we let $q \ra \infty$ to give 
    \begin{equation*}
        \frac{1}{g(q)}\sum_{m \leq g(q)}(h_q^{-1/2}S_h(m,q))^{r} \ra \mu_r,
    \end{equation*}as $q \ra \infty$, which by method of moments finishes the proof. 
\end{proof}

One might hope that somehow the state of the art of distributional results should reflect the state of bounds on the actual size of the least NQR, in which case we should be able to deal with $g(q) \approx q^{1/4},$ following the Burgess' bound \cite{Burgess_1957}. In the argument as above, an improvement of this type would be used in \eqref{non_squares_eqn}. Currently, it is bounded using Number Theory Fact \ref{incomplete_sum_lemma}, which relies on the Weil bound. Unfortunately, generalising the results of Burgess to a case where the sum of characters runs over $f(n)$ instead of $n$, where $f$ is a squarefree polynomial is not an easy task. One can find notable results that are applicable to our situation in the works of Bourgain and Chang \cite{Bourgain_Chang} and Pierce and Xu \cite{Pierce_Xu_2020}. Unfortunately, both of these papers give bounds for \eqref{non_squares_eqn} which grow in powers of $g(q)$ along with the degree of the polynomial $f$. This would possibly allow us to set $g(q) = q^{1/2-o(1)}$, but extending Theorem \ref{main_theorem_all_primes} to $g(q)$ of smaller order remains a hard problem. 

\vspace{12pt}
\noindent {\em Acknowledgements.} The author would like to thank his supervisor Adam Harper for recommending this problem, guidance and comments on the previous versions of this paper. 

\vspace{12pt}
\noindent {\em Rights.} For the purpose of open access, the author has applied a Creative Commons Attribution (CC-BY) licence to any Author Accepted Manuscript version arising from this submission.

\bibliographystyle{amsplain}
\bibliography{refs}

\end{document}